\documentclass{article}
\usepackage{amsmath,amsthm,amssymb,amsfonts}
\usepackage[latin1]{inputenc}
\usepackage{enumerate,enumitem}
\usepackage{graphics,graphicx,epsfig,psfrag}
\evensidemargin = \oddsidemargin
\theoremstyle{plain}
\newtheorem{theorem}{Theorem}
\newtheorem*{theoremA}{Theorem A}

\newtheorem{lemma}[theorem]{Lemma}

\theoremstyle{definition}
\newtheorem{definition}[theorem]{Definition}
\newtheorem*{notations}{Notations}
\theoremstyle{remark}
\newtheorem{remark}[theorem]{Remark}
\newtheorem*{remarks}{Remarks}

\newenvironment{maliste}%
{ \begin{list}%
	{--}%
	{\setlength{\labelwidth}{30pt}%
	 \setlength{\leftmargin}{15pt}%
	 \setlength{\itemsep}{\parsep}}}%
{ \end{list} }

{ \begin{list}%
	{--}%
	{\setlength{\labelwidth}{25pt}%
	 \setlength{\leftmargin}{18pt}%
	 \setlength{\itemsep}{-\parsep}%
	 \setlength{\topsep}{-\parskip}}}%
{ \end{list} }

\def\N{{\mathbb N}}
\def\Q{{\mathbb Q}}
\def\R{{\mathbb R}}
\def\Z{{\mathbb Z}}
\def\CC{{\mathcal C}}
\def\Cinf{{\CC^{\infty}}}
\def\DD{{\mathcal D}}

\def\ZZ{{\mathcal Z}}
\def\RR{{\mathcal R}}
\def\id{{\mathrm{id}}}
\def\up{\textup}
\def\from{\colon} 
\def\res#1{\mathbin{|}{}_{#1}}
 
\def\bigbars#1{\bigl\lvert #1 \bigr\rvert} 
 
\def\lrbars#1{\left\lvert #1 \right\rvert} 
\def\Bars#1{\lVert #1 \rVert}

\def\lrBars#1{\left\lVert #1 \right\rVert} 
 
\def\th{^{\text{th}}}

\def\norm#1{\left\lVert #1 \right\rVert}

\begin{document}
\title{A connectedness result for commuting diffeomorphisms of the interval}
\author{
\begin{tabular}{c}
{Hélène \textsc{Eynard}}\\
\\
{\small{Graduate School of Mathematical Sciences}} \\
{\small{University of Tokyo}}\\
{\small{Komaba, Meguro, Tokyo 153-8914 , Japan}}\\
{\small\texttt{heynardb@ms.u-tokyo.ac.jp}}
\end{tabular}}

\date{January 4, 2010}

\maketitle

\begin{abstract}
Let $\DD^r_+ [0,1]$, $r\ge1$, denote the group of orientation-preserving $\CC^r$
diffeomorphisms of $[0,1]$. We show that any two representations of $\Z^2$ in
$\DD^r_+ [0,1]$, $r\ge2$, are connected by a continuous path of representations
of $\Z^2$ in $\DD^1_+ [0,1]$. We derive this result from the classical works by
G.~Szekeres and N.~Kopell on the $\CC^1$ centralizers of the diffeomorphisms of
$[0,1)$ which are at least $\CC^2$ and fix only $0$.
\end{abstract}

\newpage

\textbf{Acknowledgements.} This article corresponds to the third chapter of my
Ph.~D. dissertation \cite{Ey2}. I am deeply grateful to my advisor, Emmanuel
Giroux, for all the discussions we had on the subject, and for the patience and
pedagogy he showed on this occasion and many others. I am very sensible to his
availability and to his continued interest in my work. I would also like to
thank Patrice Le Calvez for thoroughly reading the first version of this text
and for encouraging me to emphasize its result. This article was written during
a one year stay in Tokyo where I was kindly invited by Pr. Takashi Tsuboi, with
the financial support of the Japan Society for the Promotion of Science.\bigskip

Let $\DD^r_+ [0,1]$, $r\ge1$, denote the group of orientation-preserving $\CC^r$
diffeomorphisms of the interval $[0,1]$. For any manifold $M$, the homomorphisms
of $\pi_1 M$ to $\DD^r_+ [0,1]$, viewed as holonomy representations ---~and so,
called below representations in $\DD^r_+ [0,1]$~---, describe $\CC^r$ foliations
of codimension~$1$ on $M \times [0,1]$ that are transverse to the $[0,1]$~factor
and tangent to the boundary. In order to understand the topology of codimension
$1$ foliation spaces, we have therefore to study the topology of the space of 
representations of finitely presented/generated groups in $\DD^r_+ [0,1]$. The 
result we prove in this paper is a first step in this direction. Applied to the 
holonomy of torus leaves, it plays a crucial role in our work on deformations of
codimension~$1$ foliations on $3$-manifolds~\cite{Ey2}. The statement is: 

\begin{theoremA} \label{t:repres}
For $r \ge 2$, any two representations of $\Z^2$ in $\DD^r_+ [0,1]$ can be
connected by a continuous path of representations of $\Z^2$ in $\DD^1_+ [0,1]$.
\end{theoremA} 

Note that a representation of $\Z^2$ in $\DD^r_+ [0,1]$ is nothing but a pair of
commuting $\CC^r$ diffeomorphisms. Thus, the space of these representations will
be regarded here as the subspace 
$$ \RR^r = \bigl\{(f,g) \in (\DD^r_+ [0,1])^2 \mid f \circ g = g \circ f\bigr\} 
   \subset (\DD^r_+ [0,1])^2 $$
equipped with the induced topology, where $\DD^r_+ [0,1]$ is endowed with the 
usual $\CC^r$ topology. The question we must answer can be phrased as follows:

\medskip

\emph{Given two commuting diffeomorphisms $f, g \in \DD^r_+ [0,1]$, how to 
connect the pairs $(f,g)$ and $(\id,\id)$ by a continuous path $t \in [0,1] 
\mapsto (f_t, g_t)$ in $\RR^1$\,?}

\medskip

The key tools to handle this problem are classical results due to G.~Szekeres 
and N.~Kopell \cite{Sz,Ko} (see Sections 1 and~2) about the $\CC^1$ centralizers
of the diffeomorphisms of $[0,1)$ which are at least $\CC^2$ and fix only~$0$. 
More precisely, our construction proceeds as follows. The common fixed points of
$f$ and $g$ form a closed set whose complement is a countable union of disjoint 
open intervals $(a,b) \subset [0,1]$. On each of these, Kopell's Lemma (Theorem 
\ref{t:Kopell}) shows that $f$ (resp. $g$) either coincides with the identity or
has no fixed point (Lemma \ref{l:i/r-F}). Then the works by Szekeres and Kopell 
imply that the restrictions of $f$ and $g$ to $[a,b]$ belong either to a common 
$\CC^1$ flow or to a common infinite cyclic group generated by some $\CC^r$ 
diffeomorphism of $[a,b]$ (Lemma \ref{l:i/r-F}). In either case, it is easy to 
define the desired pair $(f_t, g_t)$ on $[a,b]$ (Lemma \ref{l:connexite}). Our 
main contribution in this article is to prove that all these pairs of local 
$\CC^1$ diffeomorphisms fit together to yield a continuous path in $\RR^1$ 
(Lemma \ref{l:connexite}). A useful tool at this point is a result of F.~Takens 
\cite{Ta} (Theorem \ref{t:Takens}) which shows that many adjacent subintervals 
$(a,b)$ can be merged and treated as a single piece (Lemma \ref{l:Takens}), 
which makes it easier to check the global regularity.

\begin{remarks}
\begin{enumerate}[label=(\emph{\roman*})]
\item 
Theorem A actually extends to representations of $\Z^k$ for any $k \in \N$ with 
exactly the same proof but heavier notations.
\item
Unfortunately, Theorem A says nothing about the connectedness of $\RR^1$ because
$\CC^1$ diffeomorphisms and their $\CC^1$ centralizers are beyond the scope of 
the works of Szekeres and Kopell. On the other hand, it will become clear in the
next section that our construction does not yield in general a path in $\RR^r$.
Thus, the question whether or not $\RR^r$ is connected remains completely open 
for every integer $r$, $1 \le r \le \infty$.
\end{enumerate}
\end{remarks}

\begin{notations}
For any $\CC^k$ function $g$ on an interval $I \subset \R$, we write
\begin{equation*}
   \Bars{ g }_k
 = \sup \left\{ \bigbars{ D^mg (x) },\ 0 \le m \le k, \ x \in I \right\} 
   \in [0,+\infty]. 
\end{equation*}
Given a vector field $\nu$ on $I$, we make no difference between $\nu$ and the
function $dx(\nu)$, where $x$ is the coordinate in $I$. Thus, the pull-back
$f^*\nu$ of $\nu$ by a diffeomorphism $f$ of $I$ is the function $\nu \circ f /
Df$. Finally, we denote by $f^k$ the $k\th$ iterate of any diffeomorphism $f$ of
$I$, with $k \in \Z$.
\end{notations}

\subsection{The results of Szekeres and Kopell} \label{s:prelim}

Let $\DD^r I$, $1 \le r \le \infty$, denote the group of $\CC^r$ diffeomorphisms 
of an interval $I \subset \R$ and $\DD^r_+I$ the subgroup of those that preserve
orientation. For $1 \le k \le r$, the $\CC^k$ centralizer of an element $f \in 
\DD^r I$ is 
$$ \ZZ^k_f = \{g \in \DD^k I \,;\, g \circ f = f \circ g\} . $$
Though this may be a quite complicated object in general, works by Szekeres and 
Kopell lead to a complete understanding of the case $k=1$ and $r \ge 2$ when $I$ 
is a semi-open interval and $f$ has no interior fixed point. In this section, we
recall their results and establish bounds that we use later in our argument. The 
original references are \cite{Sz, Ko} but detailed proofs of Theorems \ref
{t:Kopell} and \ref{t:Z1} (and a lot more on the subject) can also be found in 
\cite{Na} and \cite{Yo}.

\begin{theorem}[Kopell's Lemma] \label{t:Kopell}
Let $f$ and $g$ be two commuting diffeomorphisms of $[a,b) \subset \R$ which are
of class $\CC^2$ and $\CC^1$, respectively. If $f$ has no fixed point in $(a,b)$
and $g$ has at least one, then $g = \id$.
\end{theorem}

\begin{theorem}[Szekeres, Kopell]\label{t:Z1}
Let $f \in \DD^r[a,b)$ be a diffeomorphism of $[a,b)$ fixing only $a$ and assume
$r \ge 2$. There exists a unique $\CC^1$ vector field $\nu_f^{[a,b)}$ on $[a,b)$
whose time-$1$ map exists and coincides with $f$. Moreover, $\nu_f^{[a,b)}$ is 
of class $\CC^{r-1}$ on $(a,b)$ and its flow, which is a one-parameter subgroup 
of $\DD^1 [a,b)$, equals the whole centralizer $\ZZ^1_f$.
\end{theorem}

In this theorem, the existence part is due to Szekeres, and the uniqueness part 
follows from Kopell's Lemma. The vector field $\nu_f^{[a,b)}$ will be called the
\emph{Szekeres vector field of $f$}. It can be nicely expressed in the form:
\begin{equation} \label{e:szek}
   \nu_f^{[a,b)} = \begin{cases}
   \displaystyle{ \lambda \lim_{k \to + \infty} (f^k)^*\eta_0} 
 & \text{if $f(x) < x$ for all $x \in (a,b)$,} \\
   \displaystyle{ \lambda \lim_{k \to - \infty} (f^k)^*\eta_0} 
 & \text{if $f(x) > x$ for all $x \in (a,b)$,}
\end{cases}
\end{equation}
where $\eta_0 := (f-\id) \partial_x$ is a $\CC^{r-1}$ vector field on $[a,b)$, 
and $\lambda := \frac{ \log Df(a) }{ Df(a)-1 }$, or $1$ if $Df(a) = 1$. In other
words, the proof of Szekeres' Theorem consists in showing that:
\begin{itemize}
\item
the vector fields $\lambda (f^k)^*\eta_0$ converge in the $\CC^1$ topology as 
$k$ goes to $\pm\infty$ (depending on the sign of $f-\id$), and 
\item
$f$ is the time-$1$ map of the limit vector field.
\end{itemize}
Complete proofs of these assertions can be found in \cite{Na} and \cite{Yo}).

\begin{remark}
In general, one cannot expect the sequence $(f^k)^*\eta_0$ to converge in a
stronger topology, even if the diffeomorphism (and thus every $(f^k)^*\eta_0$,
$k \in \Z$) is $\CC^\infty$. Indeed, F.~Sergeraert constructed in \cite{Se} a
$\Cinf$ diffeomorphism of $\R_+$ whose Szekeres vector field is not $\CC^2$.
\end{remark}

Expression \eqref{e:szek} leads to the following estimates which control the
Szekeres vector field in terms of the given diffeomorphism and reflect some
continuous dependence (this continuity was studied more thoroughly by J.-C.
Yoccoz in \cite{Yo}; our bounds are established using arguments similar to his).

\begin{lemma} \label{l:Szek2}
Let $f \in \DD^r[a,b)$ be a diffeomorphism of $[a,b)$ fixing only $a$, with $0
\le a < b \le 1$ and $r \ge 2$. If $\norm{ f - \id }_2 < \delta < 1$ then the
Szekeres vector field $\nu := \nu_f^{[a,b)}$ satisfies
$$ \sup_{ (a,b) } \lrbars{ \log \frac \nu {f-\id} } < u(\delta) \quad 
   \text{and} \quad
   \sup_{ [a,b) } | D\nu | < u(\delta) $$ 
for some universal continuous function $u \from [0,1)\; \to \R$ \up(independent 
of $f$, $a$ and $b$\up) vanishing at $0$. As a consequence, there exists another
universal continuous function $v$ vanishing at $0$ such that
$$ \norm{ f^t - \id }_1 < t \, v(\delta) \quad \text{for all } t \in [0,1], $$
where $\{ f^t \}_{t \in \R}$ denotes the flow of $\nu$, with $f^1 = f$. 
\end{lemma}

\begin{proof}
We consider the case of a contracting diffeomorphism $f$. Let $\eta_k$ denote 
the $\CC^{r-1}$ vector field $(f^k)^*\eta_0$ on $[a,b)$ for all $k \in \N$,
where $\eta_0 = (f - \id) \partial_x$, and define $\theta \from (a,b) \to \R$ by
$\theta := \log \frac{ \eta_1 }{ \eta_0 }$. For all $x \in (a,b)$,
\begin{equation*}
   \theta (x) = \log \frac{ f^2(x) - f(x) }{ Df(x) \, (f(x) - x) } 
 = \log \left( \int_0^1 Df \bigl( x + s (f(x)-x) \bigr) \, ds \right)
 - \log (Df(x)).
\end{equation*}
Thus, since $Df$ is $\CC^{r-1}$ and positive on $[a,b)$, the map $\theta$
extends to a $\CC^{r-1}$ map on $[a,b)$. One can also write
$$ \theta(x) = \log Df(x_0) - \log Df(x) \quad 
   \text{for some} \quad x_0 \in [f(x),x]. $$
Since $\log Df$  is $\CC^{r-1}$ with $r\ge 2$, this implies
\begin{equation*}
   |\theta(x)| \le \norm{ D \log Df }_0 |x_0 - x|
 = \norm{ \frac{ D^2f }{ Df }}_0 |x_0 - x|
 \le \frac{ \delta }{ 1-\delta } (x - f(x)),
\end{equation*}
according to our hypothesis on $\norm{ f - \id }_2$. Now
$$ \log \frac{ \eta_{k+1} }{ \eta_k } 
 = \log \frac{ (f^k)^* \eta_1 }{ (f^k)^* \eta_0 }  
 = \log \frac{ \eta_1 \circ f^k }{ \eta_0 \circ f^k }
 = \theta \circ f^k, $$
so
\begin{equation} \label{e:theta}
   \left| \log \frac{ \eta_j }{ \eta_i }(x) \right|
 = \left| \sum_{k=i}^{j-1} \theta \circ f^k \right|
 \le \frac{ \delta }{ 1-\delta } \sum_{k=i}^{j-1} (f^k(x) - f^{k+1}(x))  
 \le \frac{ \delta}{ 1-\delta }.
\end{equation}
Taking $i=0$ and $j \to \infty$, this gives the first bound of the lemma since
$$ \lambda = \frac{ \log Df(a) }{ Df(a) - 1 }
 < \frac{ |\log (1-\delta)| }{ \delta } \xrightarrow[ \delta \to 0 ]{} 1 . $$  
The second estimate relies on the following calculation, where $Lg := D^2g / Dg$
for any $\CC^2$ diffeomorphism $g$:
\begin{align*}
   D\eta_k  = D \left( \frac{ \eta_0 \circ f^k }{ Df^k } \right)
&= D\eta_0 \circ f^k + (\eta_0 \circ f^k) D \left( \frac 1 { Df^k } \right) \\
&= D\eta_0 \circ f^k - (\eta_k Df^k) \frac{ D^2 f^k }{ (Df^k)^2 }  \\
&= D\eta_0 \circ f^k - Lf^k \eta_k \\
&= D\eta_0 \circ f^k - \sum_{i=0}^{k-1} (Lf \circ f^i) \, Df^i \; \eta_k \\
&= D\eta_0 \circ f^k - \sum_{i=0}^{k-1}
   (Lf \circ f^i) (f^{i+1} - f^i) \frac{ \eta_k }{ \eta_i } .
\end{align*}
The desired bound follows easilly, using \eqref{e:theta} and the fact that 
$D\eta_0$ equals $Df - 1$.
\end{proof}

\subsection{Rational and irrational connected components}

\begin{definition} \label{d:r/i-F}
Let $f, g \in \DD^r_+ [0,1]$ be two commuting diffeomorphisms and denote by $F
\subset [0,1]$ the set of their common fixed points. Let us say that a connected
component $(a,b)$ of the open set $U = [0,1] \setminus F$ is \emph{rational} or
\emph{irrational} depending on whether or not there exist relatively prime 
integers $p, q \in \Z$ such that $f^p$ and $g^q$ coincide on $(a,b)$. 
\end{definition}

For example, a component $(a,b)$ on which $f$ or $g$ induces the identity is 
rational, for $0$ and $1$ are relatively prime.

\begin{lemma} \label{l:i/r-F}
Let $f, g \in \DD^r_+ [0,1]$ be two commuting diffeomorphisms, $F$ the set of
their common fixed points and $(a,b)$ a connected component of $U = [0,1]
\setminus F$.
\begin{maliste}
\item[0.]
If $f \res{ [a,b] }$ differs from the identity, then $f$ has no fixed point in
$(a,b)$ and thus defines two Szekeres vector fields: $\nu_f^{ [a,b)}$ on $[a,b)$
and $\nu_f^{ (a,b] }$ on $(a,b]$.

\item[1.]
If the component $(a,b)$ is rational, there exist a diffeomorphism $h
\linebreak[1] \in \linebreak[2] \DD^r_+[a,b]$ and some relatively prime integers
$p, q \in \Z$ such that $f \res{ [a,b] } = h^q$ and $g \res{ [a,b] } = h^p$.
Moreover, if $f \res{[a,b]}$ is not the identity, $h \res{ (a,b) }$ coincides
with the time-$1/q$ maps of both Szekeres vector fields $\nu_f^{ [a,b) }$ and
$\nu_f^{ (a,b] }$.

\item[2.]
If the component $(a,b)$ is irrational, $\nu_f^{ [a,b) }$ and $\nu_f^{ (a,b] }$
coincide on $(a,b)$. Thus, there is a $\CC^1$ vector field $\nu_f^{[a,b]}$ on
$[a,b]$ whose time-$1$ map is $f \res{ [a,b] }$, and $g \res{ [a,b] }$ is the
time-$\tau$ map of this vector field for some $\tau \in \R \setminus \Q$.
\end{maliste}
\end{lemma}

\begin{remark}
According to Kopell \cite{Ko}, for a generic $\CC^r$ diffeomorphism $f$ of
$[a,b]$ with no fixed points in $(a,b)$, the Szekeres vector fields $\nu_f^{
[a,b) }$ and $\nu_f^{ (a,b] }$ do not coincide on $(a,b)$. In other words,
$\nu_f^{ [a,b) }$ does not extend to a $\CC^1$ vector field on $[a,b]$. So one
really needs to handle the rational case seperately.

As for irrational components, one might think that having $\CC^r$ time-$t$ maps
for a dense subset $\Z + \tau \Z \subset \R$ of times $t$ would force a vector
field to be $\CC^{r-1}$. But this is not true, according to \cite{Ey1}. Thus,
the diffeomorphisms obtained with our method (cf. Lemma \ref{l:connexite}) are
only $\CC^1$ in general.
\end{remark}

\begin{proof}
Suppose $f$ has a fixed point $c$ in $(a,b)$. The sequence $(g^n(c))_{n \in \Z}$
stays in $(a,b)$, consists of fixed points of $f$ (for $f$ and $g$ commute) and
is monotone (for $c$ cannot be a fixed point of both $f$ and $g$ by definition
of $(a,b)$). Thus, this sequence converges at both ends towards points which
necessarily lie in $F$, and hence are $a$ and~$b$. Therefore $g$ has no fixed 
point in $(a,b)$ and Kopell's Lemma (Theorem \ref{t:Kopell}) shows that $f$ is 
the identity on $[a,b]$, which concludes the first point.

If $(a,b)$ is a rational component, there exist relatively prime numbers $p, q
\in \Z$ such that $f^p$ and $g^q$ coincide on $[a,b]$. Then, writing $h = (f^s
g^r) \res{ [a,b] }$ where $pr + qs = 1$, $r, s \in \Z$, one gets the desired
relations $f \res{ [a,b] } = h^q$ and $g \res{ [a,b] } = h^p$. If $f \res{ [a,b]
}$ is not the identity, Theorem \ref{t:Z1} assures that the restrictions of $f$,
$g$ and $h$ to $[a,b)$ (resp. $(a,b]$) belong to the flow of the Szekeres vector
field $\nu_f^{ [a,b) }$ (resp. $\nu_f^{ (a,b] }$). The corresponding time for
$h$ is $1/q$ since $h^q = f \res{ [a,b] }$.

Now suppose the component $(a,b)$ is irrational. Denote by $\{ f_a^t \}_{t \in
\R}$ and $\{f_b^t\}_{t \in \R}$ the flows of $\nu_f^{ [a,b) }$ and $\nu_f^{
(a,b] }$, respectively, and fix a point $c \in (a,b)$. The diffeomorphisms 
$\psi_a, \psi_b \from \R \to (a,b)$ defined by $\psi_a(t) = f_a^t (c)$ and 
$\psi_b(t) = f_b^t (c)$, respectively, conjugate $f_a^\tau \res{ (a,b)}$ and 
$f_b^\tau \res{ (a,b)}$, $\tau \in \R$, to the translation $T_\tau \from t 
\mapsto t + \tau$:
\begin{equation*}
   T_\tau = \psi_a^{-1} \circ f_a^\tau \circ \psi_a
 = \psi_b^{-1} \circ f_b^\tau \circ \psi_b.
\end{equation*}
In particular,
\begin{equation*}
   T_1 = \psi_a^{-1} \circ f \circ \psi_a
 = \psi_b^{-1} \circ f \circ \psi_b,
\end{equation*}
so 
$$ T_1 = (\psi_a^{-1} \circ \psi_b) \circ T_1 \circ 
   (\psi_b^{-1} \circ \psi_a). $$
In other words, $\psi_b^{-1} \circ \psi_a$ is a diffeomorphism of $\R$ which
commutes with the unit translation $T_1$.

According to Theorem \ref{t:Z1}, $g \res{ [a,b)}$ (resp. $g \res{ (a,b] }$)
coincides with $f_a^{\tau_a}$ (resp. $f_b^{\tau_b}$) for some time $\tau_a$
(resp. $\tau_b$). But then
\begin{equation*}
   T_{\tau_a} = \psi_a^{-1} \circ g \circ \psi_a \quad \text{and} \quad 
   T_{\tau_b} = \psi_b^{-1} \circ g \circ \psi_b
\end{equation*}
so 
$$ T_{\tau_a} = (\psi_a^{-1} \circ \psi_b) \circ T_{\tau_b} \circ 
   (\psi_b^{-1} \circ \psi_a) . $$
All four diffeomorphisms of $\R$ in this last equality commute with the unit
translation. Therefore, invariance of the rotation number under conjugacy
implies that $\tau_a = \tau_b =: \tau$. This number has to be irrational, for if
$\tau = p/q$ then $f^p$ clearly coincides with $g^q$ on $(a,b)$. But if so the
diffeomorphism $\psi_b^{-1} \circ \psi_a$ commutes with both the unit
translation and an irrational translation, and hence it must itself be a
translation. Since it fixes the origin (by construction), it is in fact the
identity. This means that the flows of $\nu_f^{ [a,b) }$ and $\nu_f^{ (a,b] }$
coincide on $(a,b)$, so these vector fields are equal on $(a,b)$.
\end{proof}

We will now see (Lemma \ref{l:Takens}) that the type of the components of $[0,1]
\setminus F$ is in fact constant on the components of $[0,1] \setminus F_0$,
where $F_0 \subset F$ is the set where both $f$ and $g$ are $\CC^r$-tangent to
the identity. This is a straightforward consequence of a theorem of Takens \cite
{Ta} (extended by Yoccoz in the finite differentability case \cite{Yo}) which 
can be stated as follows:

\begin{theorem}[Takens] \label{t:Takens}
Let $f \in \DD^r_+ (a,b)$, $r \ge 2$, be a diffeomorphism with a unique fixed
point $c$. If $f$ is not $\CC^r$-tangent to the identity at $c$, the Szekeres
vector fields $\nu_f^{(a,c]}$ and $\nu_f^{[c,b)}$ are $\CC^{r-1}$ and fit
together to yield a $\CC^{r-1}$ vector field on $(a,b)$ \up(whose time-$1$ map 
is $f$\up). Furthermore, any $\CC^r$ diffeomorphism of $(a,b)$ commuting with 
$f$ and fixing $c$ coincides with the time-$\tau$ map of this vector field for 
some $\tau \in \R$.
\end{theorem}

\begin{lemma} \label{l:Takens}
Let $f, g \in \DD^r_+ [0,1]$ be two commuting diffeomorphisms, $F$ the set of 
their common fixed points, $F_0 \subset F$ the subset of those where both $f$ 
and $g$ are $\CC^r$-tangent to the identity, and $(a,b)$ a connected component 
of $U_0 = [0,1] \setminus F_0$.
\begin{maliste}
\item[1.]
If $(a,b)$ contains a rational component of $U = [0,1] \setminus F$, there
exists a diffeomorphism $h \in \DD^r_+ [a,b]$ and relatively prime integers $p,
q \in \Z$ such that $f \res{ [a,b] } = h^q$ and $g \res{ [a,b] } = h^p$.

\item[2.]
If $(a,b)$ contains an irrational component of $U$, there exists a vector field
$\nu_f^{ [a,b] }$ of class $\CC^1$ on $[a,b]$, $\CC^{r-1}$ on $(a,b)$,
$\CC^1$-flat at the boundaries and whose time-$1$ and $\tau$ maps, for some
$\tau \in \R \setminus \Q$, are $f \res{ [a,b] }$ and $g \res{ [a,b] }$,
respectively.
\end{maliste}
\end{lemma}

From now on, a \emph{rational} (resp. \emph{irrational}) component will be any
component of $U_0 = [0,1] \setminus F_0$ which contains (only) rational (resp.
irrational) components of $[0,1] \setminus F$.

\subsection{Regularity at the degenerated fixed points}

Theorem A is a straightforward consequence of the following lemma.

\begin{lemma} \label{l:connexite}
Let $f, g \in \DD^{r}_+[0,1]$ be two commuting diffeomorphisms, $F_0$ the set of
the common fixed points where both are $\CC^r$-tangent to the identity, $U_0$ 
the complement $[0,1] \setminus F_0$ and $\nu$ the vector field on $[0,1]$ equal
to $\nu_f^{[a,b]}$ on the closure of any irrational component $(a,b)$ of $U_0$ 
and to $0$ elsewhere.  For any $t \in [0,1]$, define $f_t, g_t \from [0,1] \to 
[0,1]$ as follows\up:
\begin{maliste}
\item
on $F_0$, set $f_t = f = \id$ and $g_t = g = \id$\up;
\item
on each rational component of $U_0$ where $f = h^q$ and $g = h^p$ for relatively
prime integers $p, q \in \Z$, set $f_t = h_t^q$ and $g_t = h_t^p$, where $h_t = 
(1-t)h + t\id$\up;
\item
on each irrational component of $U_0$ where $f$ and $g$ coincide with the flow
of $\nu$ at times $1$ and $\tau \in \R \setminus \Q$, respectively, let $f_t$ 
and $g_t$ be the flow maps of $\nu$ at times $(1-t)$ and $(1-t) \tau$, 
respectively.
\end{maliste}
Then the arc $t \in [0,1] \mapsto (f_t,g_t)$ is a continuous path in the space 
$\RR^1 \subset (\DD^1_+ [0,1])^2$ of commuting $\CC^1$ diffeomorphisms.
\end{lemma}

\begin{proof}
The maps $f_t$ and $g_t$ clearly commute for all $t \in [0,1]$. Since $f$ and 
$g$ play symmetric roles in the construction, it is sufficient to prove that $t 
\mapsto f_t$ is a continuous path in $\DD^1_+[0,1]$. First observe that, for all 
$t \in [0,1]$, the map $f_t$ is a homeomorphism of $[0,1]$ and, according to
Lemma \ref{l:Takens}, induces a $\CC^{r-1}$ diffeomorphism of $U_0$. Now let us
prove that $f_t$ is differentiable at any point $c$ of $F_0$, with derivative 
equal to $1$.

Let $c \in F_0$. By construction,
$$ |f_t(x) - x| \leq |f(x) - x| \quad 
   \text{for all } (t,x) \in [0,1] \times [0,1]. $$
If $x \neq c$, 
\begin{align*}
   \lrbars{ \frac{ f_t(x) - f_t(c) }{ x-c } - 1 }
&= \lrbars{ \frac{ (f_t(x)-x) - (f_t(c)-c) }{ x-c } } \\
&= \lrbars{ \frac{ f_t(x)-x }{ x-c } } \\
& \le \lrbars{ \frac{ f(x)-x }{ x-c }} \\
&= \lrbars{ \frac{ (f(x)-x) - (f(c)-c) }{ x-c } }.
\end{align*}
Since $c \in F_0$, the map $f-\id$ is $\CC^1$-flat at $c$, so the above quantity
tends to $0$ as $x$ goes to $c$. Therefore, $f_t$ admits a derivative at $c$
which is equal to $1$.

\medskip

To conclude the proof, one needs to check that the map $\Psi \from (t,x) \mapsto
Df_t(x)$ (now well-defined) is continuous on $[0,1] \times [0,1]$. For every
component $(c,d)$ of $U_0$, continuity on $[0,1] \times [c,d]$ follows readily 
from Lemma \ref{l:Takens}. In particular, the limit of $Df_s(x)$ as $(s,x)$ 
approaches $(t,c)$ in $[0,1] \times [c,1]$ exists and equals $Df_t(c) = 1$. Now
assume that $c$ is an accumulation point of $F_0$ on the right side. For all
$\delta > 0$, there exists a point $d \in F_0 \cap (c,1]$ such that
$$ \lrBars{ f - \id \res{ [c,d] } }_2 < \delta. $$
Let $(s,x) \in [0,1] \times [c,d]$. If $x$ belongs to $F_0$ then $Df_s(x) = 1$. 
If $x$ belongs to an irrational component $(a,b)$ of $U_0 \cap [c,d]$  then $f_s
\res{ [a,b] }$ belongs to the flow of $\nu$, so $\nu \circ f_s(x) = Df_s(x) \,
\nu(x)$. If $x$ is not a fixed point of $f$ (\emph{i.e.} a zero of $\nu$) then
\begin{align*}
   |Df_s(x) - 1| &= \left|\frac{ \nu \circ f_s(x) - \nu(x) }{ \nu(x) }\right| \\
&\le \sup_{ (a,b) } |D\nu| \left|\frac{ f_s(x) - x }{ \nu(x) }\right| \\
&\le \sup_{ (a,b) } \left|D\nu_f^{[a,b]}\right| \, 
   \left|\frac{ f(x) - x }{ \nu(x) }\right|
 \le u(\delta) e^{u(\delta)} \quad 
   \text{according to Lemma \ref{l:Szek2}}.
\end{align*}
This upper bound still holds for all $x \in [a,b]$ since one already knows that
$\Psi$ is continuous on $[0,1] \times [a,b]$.

Assume now that $x$ belongs to a rational component $(a,b)$ of $U_0 \cap [c,d]$
where $f = h^q$ and $g = h^p$, with $p, q \in \Z$ relatively prime. If $q$ is
zero, $f_s = \id$ on $[a,b]$ so $Df_s(x) = 1$. If $q$ is nonzero, Lemma \ref
{l:Szek2} bounds $\norm{ h-\id }_1$ by $\frac1q v(\delta)$. In particular, for 
$\delta$ small enough, $\norm{ h-\id }_1 < 1/2$. Since $f_s = h_s^q$ on $[a,b]$,
\begin{align*}
   \lrbars{ \log Df_s(x) } = \lrbars{ \log D(h_s^q)(x) } 
&= \lrbars{ \sum_{i=0}^{q-1} \log Dh_s \bigl(h_s^i (x)\bigr) } \\
&\le \sum_{i=0}^{q-1} \lrbars 
   {\log \biggl(1 + (1-s) \Bigl(Dh \bigl(h_s^i(x)\bigr) - 1\Bigr)\biggr) } \\
&\le \sum_{i=0}^{q-1} 2 (1-s) \norm{ h-\id}_1 \le 2 v(\delta) . 
\end{align*}
Thus, $Df_s(x)$ tends to $1 = Df_s(c)$ as $(s,x)$ goes to $(t,c)$ in $[0,1] 
\times [c,1]$. Similarly, $Df_s(x)$ tends to $1$ as $(s,x)$ goes to $(t,c)$ in 
$[0,1] \times [0,c]$. This proves the continuity of $\Psi$ at every point in 
$[0,1] \times F_0$, and thus on the whole of $[0,1] \times [0,1]$.
\end{proof}

\end{document}